\documentclass{amsart}
\usepackage[utf8]{inputenc}

\usepackage[boxsize=1em]{ytableau}
\usepackage{amsmath,amssymb,amsthm,mathtools,mathdots,nicefrac,tikz-cd,bbding,stmaryrd,float,multicol,caption,bbm, upgreek}
\usepackage[mathscr]{euscript}
\usepackage[margin=1in]{geometry}
\usepackage{color}
\usepackage[all]{xy}

\usepackage[shortalphabetic]{amsrefs}
\usepackage{hyperref}
\hypersetup{colorlinks=true,citecolor=purple,linkcolor=purple}

\theoremstyle{theorem}
\newtheorem{theorem}{Theorem}[section]

\theoremstyle{definition}
\newtheorem{corollary}[theorem]{Corollary}

\theoremstyle{definition}
\newtheorem{proposition}[theorem]{Proposition}

\theoremstyle{definition}
\newtheorem{lemma}[theorem]{Lemma}

\theoremstyle{definition}
\newtheorem{definition}[theorem]{Definition}

\theoremstyle{definition}

\theoremstyle{definition}
\newtheorem{remark}[theorem]{Remark}

\theoremstyle{definition}

\theoremstyle{definition}

\theoremstyle{definition}

\renewcommand{\Vec}{\mathrm{Vec}}
\newcommand{\End}{\text{End}}
\newcommand{\Flag}{\mathrm{Flag}}

\newcommand{\bfG}{\mathbf{G}}

\newcommand{\bfU}{\mathbf{U}}
\newcommand{\bfV}{\mathbf{V}}
\newcommand{\bfW}{\mathbf{W}}

\newcommand{\scrU}{\mathscr{U}}
\newcommand{\scrD}{\mathscr{D}}

\newcommand{\C}{\mathbf{C}}
\newcommand{\N}{\mathbf{N}}

\renewcommand{\P}{\mathbf{P}}
\newcommand{\GL}{\mathbf{GL}}

\newcommand{\frakS}{\mathfrak{S}}
\newcommand{\FI}{\mathbf{FI}}
\newcommand{\FB}{\mathbf{FB}}

\newcommand{\bfS}{\textbf{S}}

\newcommand{\bfM}{\textbf{M}}
\newcommand{\bfI}{\textbf{I}}
\newcommand{\bfJ}{\textbf{J}}

\newcommand{\bfL}{\textbf{L}}
\newcommand{\bfP}{\textbf{P}}
\newcommand{\bfQ}{\textbf{Q}}

\newcommand{\id}{\text{id}}

\newcommand{\Rep}{\text{Rep}}

\newcommand{\Hom}{\text{Hom}}
\newcommand{\Mod}{\text{Mod}}

\newcommand{\Ext}{\text{Ext}}

\newcommand{\interior}[1]{%
  {\kern0pt#1}^{\mathrm{o}}%
}

\DeclarePairedDelimiter\abs{\lvert}{\rvert}

\title{Polynomial functors on flags}
\author{Teresa Yu}
\address{Department of Mathematics, University of Michigan, Ann Arbor, MI}
\email{\href{mailto:twyu@umich.edu}{twyu@umich.edu}}
\urladdr{\url{https://sites.google.com/view/teresayu}}
\thanks{TY was supported by NSF grant DGE-1841052.}

\begin{document}
\maketitle

\begin{abstract}
    We study generalizations of Schur functors from categories consisting of flags of vector spaces. We give different descriptions of the category of such functors in terms of representations of certain combinatorial categories and infinite rank groups, and we apply these descriptions to study polynomial representations and representation stability of parabolic subgroups of general linear groups.
\end{abstract}

\section{Introduction}

Schur functors and polynomial functors are fundamental objects due to their ability to translate properties across fields of math such as representation theory, commutative algebra, algebraic geometry, and combinatorics. There has been recent interest in studying variants of Schur functors from categories of linear algebraic data, such as the category whose objects are vector spaces with symmetric bilinear forms; these variants provide a new perspective for studying representations of classical groups \cites{Pa,PS,SSschurgen}. In this paper, we take a similar approach by studying generalizations of Schur functors from categories where the linear algebraic data consists of flags of vector spaces, and we use these functors to study representations of parabolic subgroups of general linear groups. 

\subsection{Main result}
In order to state our main result, we introduce three categories. Each will be given the structure of a \emph{tensor} category, meaning it is linear and symmetric monoidal. Fix $n\geq 1$.

Let $\Flag_n$ be the following category: an object $\{V_i\}$ is a finite dimensional complex vector space $V$ equipped with a flag $0=V_0\subset V_1\subset\cdots\subset V_n=V$ of length $n$, and a morphism $f:V\to W$ is a linear map such that $f(V_i)\subset W_i$ for $i\in[n]=\{1,\ldots,n\}$. For $a\geq 0$, there is a functor $F_a:\Flag_n\to\Vec$ given by $F_a(\{V_i\})=V^{\otimes a}$, where $\Vec$ denotes the category of finite dimensional complex vector spaces. We say a functor $\Flag_n\to\Vec$ is \emph{polynomial} if it is a subquotient of a finite direct sum of $F_a$'s. The tensor product of such functors is defined by tensoring their values, and it is again polynomial.

Let $\GL=\bigcup_{i=1}^\infty\GL_i(\C)$ be the infinite general linear group, and let $\bfV$ be the defining representation. Equip $\bfV$ with a flag of length $n$ where each graded piece is also infinite dimensional. Let $\bfG$ be the subgroup of $\GL$ preserving this flag; this group can be identified with the group of invertible $n\times n$ block upper triangular matrices, where each block is an element of $\End(\C^\infty)=\bigcup_{i=1}^\infty\End(\C^i)$ and blocks along the diagonal are elements of $\GL$. We call $\bfV$ along with the flag the \emph{standard representation} of $\bfG$, and say a representation of $\bfG$ is \emph{polynomial} if it appears as a subquotient of a finite direct sum of tensor powers of $\bfV$. The tensor product of polynomial representations as complex vector spaces is again polynomial.

Let $\mathscr{U}$ be the following category: an object is a finite set $S$ equipped with a weight function $S\to [n]$, and a morphism $f:S\to T$ is a bijection such that weights are non-decreasing. A \emph{module} over $\mathscr{U}$ is a functor $\mathscr{U}\to\Vec$ of finite length, meaning its value at any object of $\scrU$ is a finite dimensional vector space and the functor is supported on finitely many objects of $\scrU$ (up to isomorphism). A morphism of $\mathscr{U}$-modules is a natural transformation of functors, and the tensor product of $\scrU$-modules is defined by Day convolution (see \S\ref{subsec:Umod tens} for a precise statement).

Our main result is the following.

\begin{theorem}\label{thm:main}
    There is an equivalence among the following tensor categories:
    \begin{itemize}
        \item the category of polynomial functors $\Flag_n\to\Vec$;
        \item the category of polynomial representations of $\bfG$;
        \item the category of $\mathscr{U}$-modules.
    \end{itemize}
\end{theorem}
The case with $n=1$ recovers Schur--Weyl duality in the sense of the equivalence among the categories of polynomial functors $\Vec\to \Vec$, polynomial representations of $\GL$, and finite length $\FB$-modules (see \cite{SStca} for an overview of this equivalence). Our result can therefore be seen as a generalization of Schur--Weyl duality for flags.

By studying $\scrU$-modules, we are able to describe many aspects for the other models of the above category. In particular, the category is self-dual (despite not being semisimple), all objects are of finite length and have both finite projective and injective dimensions, and we give descriptions of all indecomposable projective and injective objects. We also are able to compute the first $\Ext$ groups among simple objects and describe left-exact tensor functors from the category.

\subsection{Relation to other work}
\subsubsection{The infinite symmetric group} Let $\FI^{(n)}$ denote the category of $[n]$-weighted finite sets and injections such that the weights are non-decreasing. When $n=1$, this is the classical $\FI$ category \cite{CEF}. An $\FI^{(n)}$-module is a functor $\FI^{(n)}\to\Vec$, and such functors are related to the infinite symmetric group as follows.

Let $R=\C[x_1,x_2,\ldots]$ denote the infinite polynomial ring; it has a natural action of the infinite symmetric group $\frakS_\infty$. There has been recent interest in studying modules over $R$ that are symmetric in the sense that they have a compatible action of $\frakS_\infty$ \cites{LNNRdim,LNNRreg,MR,NSsym}. Let $\mathfrak{h}_n=(x_i^n:i\geq 1)$ be the ideal of $R$ generated by $n^\text{th}$ powers of the variables. There is a natural functor from the category of $\FI^{(n)}$-modules to the category of modules over $R/\mathfrak{h}_{n+1}$ with a compatible $\frakS_\infty$-action. In future work, we plan to apply the results in this paper to analyze the categories of $\FI^{(n)}$-modules and symmetric $R/\mathfrak{h}_n$-modules. Via the results in \cite{NSsym}, understanding the category of symmetric $R/\mathfrak{h}_n$-modules is key to sudying symmetric $R$-modules in general.

\subsubsection{Representation stability}
The idea behind representation stability is to consider sequences of groups that are naturally indexed in some way and to consider compatible representations of these groups. It has been especially fruitful to study these as functors out of categories, such as $\FI$-modules for the study of symmetric groups, $\textbf{VI}$- and $\textbf{VIC}$-modules for the study of general linear groups, and other variants of Schur functors for the study of orthogonal and symplectic groups \cites{CEF,PS,SSschurgen}. 

Representation stability can also be studied using representation theory of infinite rank groups, as done in \cite{SSstabpatterns} with representations of groups such as infinite rank orthogonal and sympletic groups. Such representations have also been studied from the perspective of locally finite Lie algebras \cites{DPS,PSe,PSt}.

\subsubsection{Brauer categories and their representations}
A useful perspective for studying representations of infinite rank groups is through representations of combinatorial or diagrammatic categories. In \cite{SSstabpatterns}, the authors use categories of various types of Brauer diagrams to study representations of different infinite rank groups, such as algebraic representations of $\GL$; many of the arguments and ideas in this paper are similar to those in that paper. The same authors have also begun to develop a more general theory for representations of Brauer categories in \cite{SSBrauerI}.

The $n=2$ case in this paper is particularly similar to the category of algebraic representations of $\GL$, which is equivalent to the category of modules over the upwards walled Brauer diagram category \cite[Corollary 3.2.12]{SSstabpatterns}. Let $\bfW=\bigcup_{d\geq 0}\C^d$ be the standard representation of $\GL$. Algebraic representations of $\GL$ can also be described in terms of certain representations of $\GL(\bfW\oplus\bfW_*)$, where $\bfW_*$ is the restricted dual. Then the Lie algebra corresponding to the unipotent radical of the subgroup of block upper triangular matrices can be identified with $\bfW\otimes\bfW$, while the corresponding Lie algebra in our case can be identified with $\bfW\otimes\bfW_*$. 

\subsubsection{Twisted commutative algebras and their applications} Let $n=1$, and consider the category of all (not necessarily of finite length) functors $\scrU\to\Vec$. Then a commutative algebra object in the category is called a \emph{twisted commutative algebra}; over $\C$, these are also \emph{$\GL$-algebras}. These algebras and their modules have been well-studied \cites{Stcaspec,SSglI,SSglII}, and have also been used to study various notions of tensor rank and to give proofs of Stillman's conjecture in commutative algebra \cites{BDDE,BDES,DLL,ESS}.

\subsection{Outline} The rest of the paper is organized as follows. In \S \ref{sec:rep FBn}, we study the category of $\scrU$-modules, and in \S \ref{sec:rep models}, we study the category of polynomial representations of $\bfG$. In \S \ref{sec:equiv}, we prove Theorem~\ref{thm:main} and explicitly describe the equivalence of categories, as well as give some immediate consequences.

\subsection{Notation and conventions} Fix $n\geq 1$. Throughout, we work over the field $\C$ of complex numbers. We write $\underline{a}=(a_1,\ldots,a_n)\in\N^n$ for an $n$-tuple of nonnegative integers, and write $\abs{\underline{a}}=a_1+\cdots+a_n$; the set of all such tuples is partially ordered by the dominance order. We let $\tau(\underline{a})=(a_n,\ldots,a_1)$ denote the reverse tuple. We write $\underline{\lambda}=(\lambda^1,\ldots,\lambda^n)$ for an $n$-tuple of partitions, where $\lambda^i$ is a partition of $\abs{\lambda^i}$.
\begin{align*}
    \Vec\textbf{: }&\text{the category of finite dimensional vector spaces}\\
    \Flag_n\textbf{: }&\text{the category of finite dimensional vector spaces with flags of length $n$}\\
    \mathscr{A}\textbf{: }&\text{the category of polynomial functors $\Flag_n\to\Vec$}\\
    \bfG\textbf{: }&\text{the subgroup of $\GL$ of invertible $n\times n$ block upper triangular matrices}\\
    \Rep(\bfG)\textbf{: }&\text{the category of polynomial representations of $\bfG$}\\
    \bfL,H_{\underline{a}},G_{\underline{a}}\textbf{: }&\text{subgroups of $\bfG$}\\
    \bfV,\{\bfV_i\}_{i=0}^n,\bfV_{(i)}\textbf{: }&\text{the standard $\bfG$-representation, its natural flag structure, and graded pieces}\\
    \scrU,\scrD\textbf{: }&\text{the categories of $[n]$-weighted finite sets and non-decreasing/increasing bijections}\\
    \Mod_\scrU,\Mod_\scrD\textbf{: }&\text{the categories of $\scrU$- and $\scrD$-modules}\\
    \bfM_{\underline{\lambda}},\P_{\underline{a}},\P_{\underline{\lambda}},\bfI_{\underline{a}},\bfI_{\underline{\lambda}}\textbf{: }&\text{objects of $\Mod_{\scrU}$}\\
    \bfQ_{\underline{a}},\bfQ_{\underline{\lambda}},\bfJ_{\underline{a}},\bfJ_{\underline{\lambda}}\textbf{: }&\text{objects of $\Mod_\scrD$}\\
    S_{\underline{\lambda}}, T_{\underline{a}}, T_{\underline{\lambda}}, K_{\underline{a}},U_{\underline{a}},U_{\underline{\lambda}}\textbf{: }&\text{objects of $\Rep(\bfG)$}\\
    \bfS_{\lambda}(-)\textbf{: }&\text{the Schur functor corresponding to the partition $\lambda$}\\
    \frakS_{\underline{a}}\textbf{: }&\text{the product of symmetric groups $\frakS_{a_1}\times\cdots\times\frakS_{a_n}$}\\
    \C^{\underline{a}}\textbf{: }&\text{the flag $0\subset\C^{a_1}\subset\C^{a_1+a_2}\subset\cdots\subset\C^{\abs{\underline{a}}}$}
\end{align*}

\subsection*{Acknowledgements} The author is grateful to Andrew Snowden for his guidance and for valuable discussions, and to Karthik Ganapathy for helpful discussions and comments. 

\section{Representations of categories of weighted finite sets}\label{sec:rep FBn}

In this section, we study representations of certain combinatorial categories of weighted finite sets. We refer the reader to \cite[\S 2.1]{SSstabpatterns} and \cite[\S 3]{SSBrauerI} for further background on representations of combinatorial categories.

\subsection{Preliminaries}

\begin{definition}
    The category of \emph{$[n]$-weighted finite sets and upwards bijections}, denoted $\scrU$, is the following category:
    \begin{itemize}
        \item The objects are $[n]$-weighted finite sets, i.e., finite sets $S=\bigsqcup_{i=1}^n S_i$ where $S_i$ has weight $i$; an object is denoted by the $n$-tuple $S=(S_1,\ldots,S_n)$.
    \item A morphism $S\to T$ is a bijection $\varphi:S\to T$ of sets such that weights do not decrease, i.e., $\varphi(S_i)\subset\bigsqcup_{j=i}^n T_j$.
    \end{itemize}
\end{definition}

For an $n$-tuple of nonnegative integers $\underline{a}\in\N^n$, there is a weighted finite set $([a_1],\ldots,[a_n])$, where $[a_i]$ has weight $i$. Every object of $\scrU$ is isomorphic to a unique such object. We will use $\underline{a}$ to denote both the tuple of integers and the corresponding weighted finite set; context should clarify any confusion. 

There is also the category of \emph{$[n]$-weighted finite sets and downwards bijections}, denoted $\scrD$. Its definition is the same, except that a morphism is a bijection of finite sets such that weights do not increase, and so $\scrD$ can naturally be identified with $\scrU^{\text{op}}$.

Both $\scrU$ and $\scrD$ have a symmetric monoidal structure $\amalg$ defined by disjoint union: multiplication is defined by $S\amalg T=(S_1\sqcup T_1,\ldots,S_n\sqcup T_n),$ and the empty set $\varnothing$ is the unit object. The automorphism group of an object $S$ in either category is the product of symmetric groups $\text{Aut}(S_1)\times\cdots\times\text{Aut}(S_n)$. 

A \emph{$\scrU$-module} $M$ is a functor $\scrU\to\Vec$ that is supported on finitely many objects of $\scrU$ up to isomorphism. 
For a morphism $\varphi:S\to T$ in $\scrU$, we write $\varphi_*:M(S)\to M(T)$ to denote the linear map $M(\varphi)$. A \emph{morphism} of $\scrU$-modules is a natural transformation of functors. Let $\Mod_\scrU$ denote the category of $\scrU$-modules. This is an abelian category. For $\scrU$-modules $M$ and $N$, we write $\Hom_\scrU(M,N)$ for the space of morphisms $M\to N$. We also write $\Hom_\scrU(S,T)$ for the space of morphisms $S\to T$ for $S,T\in\scrU$. 

Analogous definitions and notation hold for modules over $\scrD$.

\subsection{Duality}

Given an $\scrU$-module $M$, we define the $\scrD$-module $M^\vee$ by $M^\vee(S)=M(S)^*$, where $(-)^*$ denotes the dual vector space. This induces an exact functor $(\Mod_\scrU)^{\text{op}}\to\Mod_\scrD$, which gives an equivalence of categories via the canonical isomorphism $M\to (M^\vee)^\vee$.

There is also a (covariant) equivalence of monoidal categories $\scrU\cong\scrD$ given by changing the weights of elements from $i$ to $n-i+1$:
    \[\tau:S=(S_1,\ldots,S_n)\mapsto(S_n,S_{n-1},\ldots,S_1).\]
    A morphism $\varphi:S\to T$ in $\scrU$ also gives a morphism $(S_n,\ldots,S_1)\to (T_n,\ldots,T_1)$ in $\scrD$, since now $\varphi$ does not increase weights. For an $n$-tuple $\underline{a}$ or $\underline{\lambda}$, we also use $\tau(\underline{a})$ or $\tau(\underline{\lambda})$ to denote the reverse tuple.

As a consequence, we have a covariant equivalence of categories $\Mod_\scrU\cong\Mod_\scrD$ given by the pushforward and pullback of $\tau$:
\[\tau_!:\Mod_\scrU\to\Mod_\scrD,\qquad\tau^*:\Mod_\scrD\to\Mod_\scrU.\] 
If $M$ is a $\scrU$-module, then $\tau_!M$ is defined by $\tau_!M(S_1,\ldots,S_n)=M(S_n,\ldots,S_1)$. If $\varphi:S\to T$ is a morphism in $\scrD$, then the linear map $\tau_!M(S)\to\tau_!M(T)$ is defined to be the map $\varphi_*=M(\varphi):M(S_n,\ldots,S_1)\to M(T_n,\ldots,T_1)$. The pullback $\tau^*$ is defined similarly.

\begin{proposition}\label{prop:ModFBn selfdual}
    The functors $\tau_!,\tau^*$ give an equivalence of categories $\Mod_\scrU\cong\Mod_\scrD$. In particular, this category is self-dual.
\end{proposition}

\subsection{Simple, projective, and injective modules}

For a partition $\lambda$ with $\abs{\lambda}=m$, let $M_\lambda$ denote the Specht module corresponding to $\lambda$; this is a simple representation of $\frakS_m$. Then, the simple $\scrU$-modules and $\scrD$-modules are indexed by $n$-tuples of partitions $\underline{\lambda}=(\lambda^1,\ldots,\lambda^n)$, and are defined by
\[S\mapsto \begin{cases}
    M_{\lambda^1}\boxtimes\cdots\boxtimes M_{\lambda^n}&\quad\text{if $\abs{\lambda^i}=\abs{S_i}$ for all $i$,}\\
    0&\quad\text{otherwise.}
\end{cases}\]
We let $\bfM_{\underline{\lambda}}$ denote this module, as well as the corresponding irreducible representation of $\frakS_{\underline{a}}=\frakS_{a_1}\times\cdots\times\frakS_{a_n}$, where $a_i=\abs{\lambda^i}$.

We define the $\scrU$-module $\P_{\underline{a}}$, called the \emph{principal projective} at $\underline{a}$, by $\P_{\underline{a}}(T)=\C[\Hom_\scrU(\underline{a},T)].$
We define the $\scrU$-module $\bfI_{\underline{a}}$, called the \emph{principal injective} at $\underline{a}$, by $\bfI_{\underline{a}}(T)=\C[\Hom_{\scrU}(T,\underline{a})]^*.$
Let $M$ be a $\scrU$-module. By \cite[Proposition 3.2]{SSBrauerI}, we have that
\[\Hom_\scrU(\P_{\underline{a}},M)=M(\underline{a}), \qquad \Hom_\scrU(M,\bfI_{\underline{a}})=M(\underline{a})^*.\]
In particular, the principal projectives are indeed projective objects, and the principal injectives are injective objects.

We have that $\frakS_{\underline{a}}$ acts on $\P_{\underline{a}}$ and $\bfI_{\underline{a}}$ by $\scrU$-module automorphisms, and so for a tuple of partitions $\underline{\lambda}$ with $\abs{\lambda^i}=a_i$ for all $i$, we have $\scrU$-modules $\P_{\underline{\lambda}}$ and $\bfI_{\underline{\lambda}}$ defined to be the $\bfM_{\underline{\lambda}}$-isotypic pieces of $\P_{\underline{a}}$ and $\bfI_{\underline{a}}$:
\[\P_{\underline{\lambda}}=\Hom_{\frakS_{\underline{a}}}(\bfM_{\underline{\lambda}},\P_{\underline{a}}),\qquad \bfI_{\underline{\lambda}}=\Hom_{\frakS_{\underline{a}}}(\bfM_{\underline{\lambda}},\bfI_{\underline{a}}).\]

\begin{proposition}\label{prop:fb indecomposable injs}
    The module $\bfI_{\underline{\lambda}}$ is an indecomposable injective, and its socle is $\bfM_{\underline{\lambda}}$.
\end{proposition}

\begin{proof}
    First, observe that $\bfI_{\underline{\lambda}}$ is a direct summand of the injective $\bfI_{\underline{a}}$, and so it is injective as well.

    For a $\scrU$-module $M$, we have that
    \[\Hom_{\scrU}(M,\bfI_{\underline{\lambda}})=\Hom_{\frakS_{\underline{a}}}(M(\underline{a}),\bfM_{\underline{\lambda}}^*).\]
    In particular, taking $M=\bfI_{\underline{\lambda}}$, we see that $\End_{\scrU}(\bfI_{\underline{\lambda}})=\Hom_{\frakS_{\underline{a}}}(\bfM_{\underline{\lambda}},\bfM_{\underline{\lambda}}^*)=\C,$
    where the last equality comes from working over $\C$ so Specht moduls are self-dual. We therefore have that there are no nontrivial idempotent endomorphisms of $\bfI_{\underline{\lambda}}$, and so it is indecomposable.

    To show that $\bfM_{\underline{\lambda}}$ is the socle, it suffices to note that $\bfM_{\underline{\lambda}}\subset\bfI_{\underline{\lambda}}$; this can be seen by
    \[\Hom_{\scrU}(\bfM_{\underline{\lambda}},\bfI_{\underline{\lambda}})=\Hom_{\frakS_{\underline{a}}}(\bfM_{\underline{\lambda}},\bfM_{\underline{\lambda}}^*)=\C.\]
    Then, $\bfI_{\underline{\lambda}}$ is the injective hull of $\bfM_{\underline{\lambda}}$.
\end{proof}

\begin{proposition}
    The module $\P_{\underline{\lambda}}$ is an indecomposable projective, and it is the projective cover of $\bfM_{\underline{\lambda}}$.
\end{proposition}

\begin{proof}
    It is a direct summand of the projective $\P_{\underline{a}}$, and so it is projective as well. We have that
    \[\End_\scrU(\P_{\underline{\lambda}})=\Hom_{\frakS_{\underline{a}}}(\bfM_{\underline{\lambda}},\P_{\underline{\lambda}}(\underline{a}))=\Hom_{\frakS_{\underline{a}}}(\bfM_{\underline{\lambda}},\bfM_{\underline{\lambda}})=\C,\]
    and so $\P_{\underline{\lambda}}$ is indecomposable.

    There is a surjection $\P_{\underline{\lambda}}\to\bfM_{\underline{\lambda}}$, since
    \[\Hom_\scrU(\P_{\underline{\lambda}},\bfM_{\underline{\lambda}})=\Hom_{\frakS_{\underline{a}}}(\bfM_{\underline{\lambda}},\bfM_{\underline{\lambda}})=\C.\]
    Therefore, by \cite[Lemma 3.6]{Kr}, the surjection is a projective cover.
\end{proof}

\begin{corollary}\label{cor:classifyig indecomp injs}
\begin{enumerate}
    \item The $\scrU$-modules $\bfI_{\underline{\lambda}}$ form a complete irredundant set of indecomposable injectives in $\Mod_{\scrU}$, and the $\scrU$-modules $\bfP_{\underline{\lambda}}$ form a complete irredundant set of indecomposable projectives in $\Mod_\scrU$.
    \item Every object in $\Mod_{\scrU}$ has finite injective and projective dimension.
\end{enumerate}
\end{corollary}

\begin{proof}
(1) This follows from all simples being of the form $\bfM_{\underline{\lambda}}$.

(2)  Suppose $\underline{a}=(\abs{\lambda^1},\ldots,\abs{\lambda^n})$. 
The quotient $\bfI_{\underline{\lambda}}/\bfM_{\underline{\lambda}}$ is supported on weighted finite sets $\underline{b}$ such that $\abs{\underline{a}}=\abs{\underline{b}}$ and $\underline{b}>\underline{a}$ under the dominance order, and the kernel of the surjection $\bfP_{\underline{\lambda}}\to\bfM_{\underline{\lambda}}$ is supported on weighted finite sets $\underline{c}$ such that $\abs{\underline{a}}=\abs{\underline{c}}$ and $\underline{c}<\underline{a}$. There are only finitely many such sets $\underline{b}$ and $\underline{c}$.
Thus, any $\scrU$-module has a finite injective resolution given by finite direct sums of indecomposable injectives, and a finite projective resolution given by finite direct sums of indecomposable projectives.
\end{proof}

The above definitions and results also hold for $\scrD$-modules; we denote the corresponding projective modules by $\bfQ_{\underline{a}},\bfQ_{\underline{\lambda}}$, and the injective modules by $\bfJ_{\underline{a}},\bfJ_{\underline{\lambda}}$. Under the equivalence of categories $(\Mod_{\scrU})^{\text{op}}\cong\Mod_{\scrD}$, we have that $(\P_{\underline{a}})^\vee=\bfJ_{\underline{a}}$ and $ (\bfI_{\underline{a}})^\vee=\bfQ_{\underline{a}}$. Under the equivalence of categories $\Mod_{\scrU}\cong\Mod_{\scrD}$, we have
$\P_{\underline{a}}=\bfQ_{\tau(\underline{a})}$ and $\bfI_{\underline{a}}=\bfJ_{\tau(\underline{a})}.$ The analogous identifications hold for the indecomposable injective and projective objects.

\subsection{Tensor products}\label{subsec:Umod tens}

Using the monoidal structure $\amalg$ of $\scrU$, one can define a tensor product of $\scrU$-modules (and analogously of $\scrD$-modules) as follows. Let $M,N$ be $\scrU$-modules. We define $M\otimes N$ by 
\begin{equation}\label{eqn:tensor}(M\otimes N)(U)=\bigoplus_{U=S\amalg T}M(S)\otimes_\C N(T),\end{equation}
where $U,S,T$ are weighted sets, and the weights of $S,T$ are induced from those of $U$.

The following lemma shows how this tensor product works on principal projectives. For $\underline{a},\underline{b}\in\N^n$, let $\underline{a+b}=(a_1+b_1,\ldots,a_n+b_n)$. Note that this is isomorphic to $\underline{a}\amalg\underline{b}$ as weighted sets.

\begin{lemma}\label{lem:tensor prin proj}
    For $\underline{a},\underline{b}\in\N^n$, there is a natural isomorphism $\P_{\underline{a}}\otimes\P_{\underline{b}}\cong\P_{\underline{a+b}}$.
\end{lemma}

\begin{proof}
Let $U\in\scrU$ be a weighted finite set. Given a morphim $\varphi:\underline{a+b}\to U$ in $\scrU$, one obtains morphisms $\varphi_1:\underline{a}\to S$, $\varphi_2:\underline{b}\to U\setminus S$ by restricting $\varphi$ via the natural identification $\underline{a+b}=\underline{a}\amalg\underline{b}$. On the other hand, given morphisms $\psi_1:\underline{a}\to S$, $\psi_2:\underline{b}\to T$ such that $S\amalg T=U$, one obtains a unique morphism $\psi:\underline{a}\amalg\underline{b}\to S\amalg T=U$. In particular, we have
\[\C[\Hom_\scrU(\underline{a+b},U)]=\bigoplus_{U=S\amalg T}\C[\Hom_\scrU(\underline{a},S)]\otimes\C[\Hom_\scrU(\underline{b},T)]\]
as vector spaces.

Furthermore, this identification is natural in the sense that if $\sigma:U\to V$ is a morphism in $\scrU$, then the maps $\P_{\underline{a+b}}(U)\to\P_{\underline{a+b}}(V)$ and $(\P_{\underline{a}}\otimes\P_{\underline{b}})(U)\to (\P_{\underline{a}}\otimes\P_{\underline{b}})(V)$ agree. Thus, we have the desired isomorphism of $\scrU$-modules.
\end{proof}

By (\ref{eqn:tensor}), this tensor product is exact in each variable. In particular, it is right-exact, and so since it satisfies the property on tensor products of principal projectives in Lemma~\ref{lem:tensor prin proj}, this tensor product agrees with \emph{Day convolution}, or the \emph{convolution tensor product} defined in \cite[(2.1.14)]{SSstabpatterns}. Therefore, this tensor product naturally gives $\Mod_\scrU$ the structure of a tensor category with unit object $\P_\varnothing$. Note also that the equivalence $\Mod_\scrU\cong\Mod_\scrD$ given by $\tau$ is one of tensor categories.

\subsection{Structured $\Hom$ spaces and tensor products}\label{subsec:structured hom}

We now give two constructions using $\scrU$-modules and $\scrD$-modules. They will be used in \S\ref{sec:equiv} to describe the equivalence of categories $\Mod_\scrU\cong\Rep(\bfG)$, as well as to give a universal property for this category.

Let $\mathscr{C}$ be an abelian category, and let $\mathcal{K}$ be a functor $\scrU\to\mathscr{C}$. For a $\scrU$-module $M$, we define an object $\Hom(M,\mathcal{K})$ of $\mathscr{C}$, called the \emph{structured Hom space}, as
\[\Hom(M,\mathcal{K})=\lim_{f:S\to T}\Hom(M(S),\mathcal{K}(S)),\]
where the inverse limit is over morphisms in $\scrU$. Note that this is a finite limit since $M$ is of finite length and $\scrU$ decomposes into finite directed categories, and so this Hom space always exists. More explicitly, the object in $\mathscr{C}$ is defined by the following mapping property. To give a morphism $f:A\to\Hom(M,\mathcal{K})$ in $\mathscr{C}$ is the same as giving morphisms $f_S:A\to \Hom(M(S),\mathcal{K}(S))$ for all $S\in\mathscr{U}$ such that for any $x\in A$ and morphism $\alpha:S\to T$ in $\mathscr{U}$, we have that the following diagram commutes:
    \[
 \xymatrix@C+0.5em@R+0.5em{ 
 	 M(S) \ar^{f_S(x)}[r] \ar^{\alpha_*}[d] & \mathcal{K}(S) \ar^{\alpha_*}[d]  \\
 	M(T) \ar^{f_T(x)}[r] & \mathcal{K}(T)
 }
 \]

 The next result follows immediately from the universal property.

\begin{lemma}\label{lem:structured hom prin proj}
    We have an identification $\Hom(\P_{\underline{a}},\mathcal{K})=\mathcal{K}(\underline{a}).$
\end{lemma}

Another useful property regarding this construction is the following. There is a functor $\mathscr{C}\to\Mod_\scrU$ defined by
\[N\mapsto(S\mapsto\Hom_\mathscr{C}(N,\mathcal{K}(S));\]
we denote the resulting $\scrU$-module by $\Hom_\mathscr{C}(N,\mathcal{K})$.
Then by \cite[Proposition 2.1.10]{SSstabpatterns}, the functor $M\mapsto\Hom(M,\mathcal{K})$ is adjoint on the right to this functor:
\[\Hom_\scrU(M,\Hom_\mathscr{C}(N,\mathcal{K}))=\Hom_\mathscr{C}(N,\Hom(M,\mathcal{K})).\]

There is a covariant version of the structured Hom construction. If $M$ is a $\scrD$-module, then we define the \emph{structured tensor product} $M\odot \mathcal{K}$ to be
\[M\odot\mathcal{K}=\Hom(M^\vee,\mathcal{K}).\]
It can also be defined by an inverse limit; see \cite[(2.1.9)]{SSstabpatterns}. By \cite[Proposition 2.1.16]{SSstabpatterns}, we have that the functors
\begin{align*}
    \Mod_\scrU\to\mathscr{C}:&\quad M\mapsto\Hom(M,\mathcal{K}),\\
    \Mod_\scrD\to\mathscr{C}:&\quad N\mapsto N\odot\mathcal{K}
\end{align*}
are both tensor functors.

\section{Representations of infinite rank parabolic subgroups}\label{sec:rep models}

In this section, we analyze the category $\Rep(\bfG)$ of polynomial representations of $\bfG$.
Such representations can be analyzed using representation theory of the Levi subgroup $\bfL$ of $\bfG$, since $\bfL\cong\prod_{i=1}^n\GL$.
Many of the arguments in this subsection are similar to those in \cite{SSstabpatterns}, e.g., their argument that the category of polynomial represetnations of the infinite orthogonal group is equivalent to the upwards Brauer category in \cite[\S 4.2]{SSstabpatterns}.

Recall that $\bfV$ is the infinite dimensional vector space with the flag $0=\bfV_0\subset\bfV_1\subset\cdots\subset\bfV_n=\bfV$ where $\bfV_i/\bfV_{i-1}$ is also infinite dimensional; this is the standard representation of $\bfG$. We let $\bfV_{(i)}$ denote the $i^{\text{th}}$ graded piece $\bfV_i/\bfV_{i-1}$. 

\subsection{Action of the endomorphism monoid}
Let $V$ be a polynomial representation of $\bfG$, and let $\rho:\bfG\to\GL(V)$ denote the map giving the $\bfG$-action. Then, after picking a basis for $V$, the entries of $\rho$ can be expressed in terms of polynomials.

Let $\End(\bfV)$ denote the monoid of endomorphisms of $\bfV$ preserving the flag structure, so an element is a linear map $f:\bfV\to \bfV$ such that $f(\bfV_i)\subset\bfV_i$. Then, any polynomial representation $V$ of $\bfG$ also has an action of $\End(\bfV)$ via the polynomial entries of the map $\rho$. Furthermore, if $V\to W$ is a map of polynomial $\bfG$-representations, then the map is equivariant with respect to the $\End(\bfV)$-actions on $V$ and $W$.

\subsection{Invariants and specialization}\label{subsec:G invar}
Let $\underline{a}\in\N^n$. Let $H_{\underline{a}}$ denote the subgroup of $\bfG$ consisting of block diagonal matrices such that the $(i,i)$ block is of the form
\[\begin{pmatrix} 1 & 0 \\ 0 & * \end{pmatrix},\]
where the top left block is $a_i\times a_i$ and the bottom right block is $(\infty-a_i)\times(\infty-a_i)$. Note that $H_{\underline{a}}\cong\bfL$. Let $G_{\underline{a}}$ denote the subgroup of $\bfG$ consisting of matrices where the $(i,j)$ block is of the form
\begin{align*}
    \begin{pmatrix}* & 0 \\ 0 & 1\end{pmatrix}&\quad\text{if $i=j$},\qquad\begin{pmatrix}* & 0 \\ 0 & 0 \end{pmatrix}\quad\text{if $i<j$},
\end{align*}
where the top left block is $a_i\times a_j$ and the bottom right block is $(\infty-a_i)\times(\infty- a_j)$ in both cases. 
We have that $G_{\underline{a}}$ commutes with $H_{\underline{a}}$, and so if $V\in\Rep(\bfG)$, then $V^{H_{\underline{a}}}$ is a representation of $G_{\underline{a}}$. Let $\Rep(G_{\underline{a}})$ denote the category of polynomial representations of $G_{\underline{a}}$.

\begin{lemma}\label{lem:invar tensor}
    Taking $H_{\underline{a}}$-invariants induces a tensor functor $\Rep(\bfG)\to\Rep(G_{\underline{a}})$.
\end{lemma}

\begin{proof}
For $n=1$, this follows from the fact that taking such invariants for a polynomial $\GL$-representation is a tensor functor \cite[Proposition 3.4.4]{SSstabpatterns}. For general $n$, we have that a polynomial $\bfG$-representation is also a polynomial $\prod_{i=1}^n\GL$-representation by restricting to the action of the Levi subgroup $\bfL$. The result then follows by considering taking $H_{\underline{a}}$-invariants as a functor from the $n$-fold product of the category of polynomial $\GL$-representations.   
\end{proof}

\subsection{Principal objects}\label{subsec:G prin obs}
For an $n$-tuple $\underline{a}\in\N^n$, we define the polynomial $\bfG$-representations $T_{\underline{a}}$ and $U_{\underline{a}}$ by
\begin{align*}
    T_{\underline{a}}&=\bigotimes_{i=1}^n (\bfV/\bfV_{i-1})^{\otimes a_i},\qquad  U_{\underline{a}}=\bigotimes_{i=1}^n \bfV_i^{\otimes a_i}.
\end{align*}
After showing the equivalence of categories in \S \ref{sec:equiv}, it will follow that $T_{\underline{a}}$ is injective and $U_{\underline{a}}$ is projective.

For a morphism $\varphi:\underline{a}\to\underline{b}$ in $\scrU$, there is a corresponding map of $\bfG$-representations $T_{\underline{a}}\to T_{\underline{b}}$ defined as follows. A tensor factor $\bfV/\bfV_{i-1}$ of $T_{\underline{a}}$ corresponds to an element $x$ of weight $i$ in $\underline{a}$, and $\varphi(x)$ has weight $j$ in $\underline{b}$ for some $j\geq i$. Then surject $\bfV/\bfV_{i-1}$ onto the tensor factor $\bfV/\bfV_{j-1}$ corresponding to $\varphi(x)$. 

Now, define the polynomial $\bfG$-representation $K_{\underline{a}}$ by
\[K_{\underline{a}}=\bigcap\ker(T_{\underline{a}}\to T_{\underline{b}}),\]
where the intersection ranges over all non-isomorphisms in $\scrU$.
This intersection is finite, as the only such morphisms are if $\abs{\underline{a}}=\abs{\underline{b}}$ with $\underline{a}>\underline{b}$.

\begin{lemma}\label{lem:homG Ta}
Fix $\underline{a}\in\N^n$.
    \begin{enumerate}
        \item The $\scrU$-module given by $\underline{b}\mapsto\Hom_{\bfG}(T_{\underline{a}},T_{\underline{b}})$ is equal to $\P_{\underline{a}}=(\bfJ_{\underline{a}})^\vee$. In particular, $\Hom_\bfG(T_{\underline{a}},T_{\underline{b}})\neq 0$ if and only if $\abs{\underline{a}}=\abs{\underline{b}}$ and $\underline{a}\geq \underline{b}$. When $\abs{\underline{a}}=\abs{\underline{b}}$,
        \[\dim (\Hom_\bfG(T_{\underline{a}},T_{\underline{b}}))=b_1!b_2!\cdots b_n!\prod_{i=1}^{n-1}\binom{(a_1+\cdots+a_i)-(b_1+\cdots+b_{i-1})}{b_i}.\]
        \item $K_{\underline{a}}=\bigotimes_{i=1}^n\bfV_{(i)}^{\otimes a_i}$.
    \end{enumerate}
\end{lemma}

\begin{proof}
(1) We first show that $\Hom_{\bfG}(T_{\underline{a}},T_{\underline{b}})=\C[\Hom_{\scrU}(\underline{a},\underline{b})]$. Let $\{e_{i,j}\}_{j\geq 1}$ be a basis for $\bfV_{(i)}$. Then $T_{\underline{a}}$ is generated as a $\bfG$-module by $x=e_{n,1}\otimes\cdots\otimes e_{n,\abs{\underline{a}}}$; this element has weight $(0^{n-1},1^{\abs{\underline{a}}})$ under the $\bfL$-action. If $\abs{\underline{a}}\neq\abs{\underline{b}}$, then there is no weight space in $T_{\underline{b}}$ of this weight, and so there are no nonzero maps $T_{\underline{a}}\to T_{\underline{b}}$.

        Now assume $\abs{\underline{a}}=\abs{\underline{b}}$ with $\underline{a}\lneq\underline{b}$. Suppose for contradiction $\varphi:T_{\underline{a}}\to T_{\underline{b}}$ is a nonzero map. Up to a scalar, the image of $x$ must be of the form $\varphi(x)=e_{n,\sigma^{-1}(1)}\otimes\cdots\otimes e_{n,\sigma^{-1}(\abs{\underline{a}})}$, where $\sigma\in\frakS_{\abs{\underline{a}}}$. Then for some $i,j$ with $1\leq i<j\leq n$, there exists $k$ with 
        \[b_1+\cdots+b_{i-1}<k\leq b_1+\cdots +b_{i},\qquad a_1+\cdots+a_{j-1}<\sigma^{-1}(k)\leq a_1+\cdots+a_j.\]
        Let $g\in\bfG$ be the element with $I$'s along the diagonal, and $e_{1,\sigma^{-1}(k)}$ in the $(i,n)$-block. Then $gx-x=0$ in $T_{\underline{a}}$, while $\varphi(gx-x)=g\varphi(x)-\varphi(x)\neq 0$ in $T_{\underline{b}}$, since $g(e_{n,\sigma^{-1}(k)})= e_{i,1}+e_{n,\sigma^{-1}(k)}$. Thus, such a nonzero $\varphi$ cannot exist.

        Now suppose $\abs{\underline{a}}=\abs{\underline{b}}$ with $\underline{a}\geq \underline{b}$. A nonzero map $\varphi:T_{\underline{a}}\to T_{\underline{b}}$ must send $x$ to an element of the form (up to scalar) $e_{n,\sigma^{-1}(1)}\otimes\cdots\otimes e_{n,\sigma^{-1}(\abs{\underline{a}})}$ for some $\sigma\in\frakS_{\abs{\underline{a}}}$. In addition, we must have that $\sigma^{-1}([b_1+\cdots+b_i])\subset[a_1+\cdots+a_i]$ for each $i=1,\ldots,n$. This is because there are no nonzero maps $\bfV/\bfV_{j-1}\to\bfV/\bfV_{i-1}$ for $i<j$ by the above argument. Thus, $\varphi$ is determined by the morphism $\sigma:\underline{a}\to\underline{b}$ in $\scrU$. Furthermore, if $\pi:\underline{b}\to\underline{c}$ is a morphism in $\scrU$, then $\pi_*\circ\varphi:T_{\underline{a}}\to T_{\underline{c}}$ corresponds to the morphism $\pi\circ \sigma$ in $\scrU$. This identifies the $\mathscr{U}$-module $\underline{b}\mapsto\Hom_\bfG(T_{\underline{a}},T_{\underline{b}})$ with the principal projective $\P_{\underline{a}}$.
        
(2) Let $j\in[n-1]$, and let $\underline{b}=(a_1,\ldots,a_{j-1},0,a_j+a_{j+1},a_{j+2},\ldots,a_n)$. Then the kernel of the surjection $\varphi_j:T_{\underline{a}}\to T_{\underline{b}}$ is
        \[\left(\bigotimes_{i=1}^{j-1}(\bfV/\bfV_{i-1})^{\otimes a_i}\right)\otimes \bfV_{(j)}^{a_j}\otimes\left(\bigotimes_{i=j+1}^n(\bfV/\bfV_{i-1})^{\otimes a_i}\right).\]
        Thus, $K_{\underline{a}}\subset\bigcap_j\ker\varphi_j=\bigotimes_{i=1}^n \bfV_{(i)}^{\otimes a_i}$.
        
        The vector
        \[v=e_{1,1}\otimes \cdots\otimes e_{1,a_1}\otimes\cdots\otimes e_{n,1}\otimes\cdots\otimes e_{n,a_n}\]
        generates $\bigotimes_{i=1}^n\bfV_{(i)}^{\otimes a_i}$ as a $\bfG$-module, and $v$ has weight $(1^{a_1},\ldots,1^{a_n})$. For any $T_{\underline{b}}$ with $\underline{a}>\underline{b}$, there is no weight space with this weight, so $\bigotimes_{i=1}^n\bfV_{(i)}^{\otimes a_i}\subset K_{\underline{a}}$.
\end{proof}

\begin{remark}\label{rmk:Tb kernel}
    If $\underline{a}>\underline{b}>\underline{c}$ with $\abs{\underline{a}}=\abs{\underline{b}}=\abs{\underline{c}}$, then any morphism $f:\underline{a}\to\underline{c}$ in $\scrU$ factors through a morphism $g:\underline{a}\to\underline{b}$. In particular, $\ker(T_{\underline{a}}\to T_{\underline{b}})\subset\ker(T_{\underline{a}}\to T_{\underline{c}})$, and so $K_{\underline{a}}$ can equivalently be given by $\bigcap\ker(T_{\underline{a}}\to T_{\underline{b}})$, where $\underline{a}>\underline{b}$ is a cover relation in the dominance order.
\end{remark}

\subsection{Simple representations}

Recall that $\bfL\cong\prod_{i=1}^n\GL$ denotes the Levi subgroup of $\bfG$, and let $\bfU=\bfG/\bfL$ denote the unipotent radical. There is a surjection of groups $\bfG\to\bfL$, so every simple $\bfL$-module pulls back to a simple $\bfG$-module. The simple $\bfL$-modules are indexed by tuples of partitions $\underline{\lambda}$, and we denote them by $S_{\underline{\lambda}}$:
\[S_{\underline{\lambda}}=\bfS_{\lambda^1}(\bfV_{(1)})\otimes\cdots\otimes\bfS_{\lambda^n}(\bfV_{(n)}),\]
where $\bfS_{\lambda^i}(-)$ denotes the Schur functor corresponding to the partition $\lambda^i$. We now show that these are precisely the simple polynomial representations of $\bfG$.

\begin{proposition}
    Every simple object of $\Rep(\bfG)$ is of the form $S_{\underline{\lambda}}$.
\end{proposition}

\begin{proof}
    It suffices to show that every simple constituent of $T_{\underline{a}}$ is of the form $S_{\underline{\lambda}}$. We have an exact sequence
    \begin{equation}\label{eqn:ses inj}
        0\to K_{\underline{a}}\to T_{\underline{a}}\to\bigoplus_{\underline{a}\to\underline{b}}T_{\underline{b}},
    \end{equation}
    where the direct sum is over all non-isomorphisms in $\scrU$ with $\abs{\underline{a}}=\abs{\underline{b}}$ and $\underline{a}>\underline{b}$ a cover relation.
    By Lemma~\ref{lem:homG Ta}, the action of $\bfG$ on $K_{\underline{a}}$ is given by the action of $\bfL$. The simple constituents of $K_{\underline{a}}$ are therefore given by its decomposition as a $\bfL$-module, so they are of the form $S_{\underline{\lambda}}$.

    Among the set $\{\underline{b}:\abs{\underline{a}}=\abs{\underline{b}}\}$, the tuple $(0,\ldots,0,\abs{\underline{a}})$ is minimal in the dominance order, and $T_{(0,\ldots,0,\abs{\underline{a}})}=(\bfV_{(n)})^{\otimes \abs{\underline{a}}}$ has simple factors given by $\bfS_{\lambda}(\bfV_{(n)})$. Thus, by induction, the simple constituents of $T_{\underline{a}}$ are of the form $S_{\underline{\lambda}}$ as well.
\end{proof}

\begin{corollary}\label{cor:ss Gmod}
    A polynomial $\bfG$-representation $V$ is semisimple if and only if $\bfU$ acts trivially on $V$.
\end{corollary}

\begin{proof}
    If $V$ is semisimple, then since $\bfU$ acts trivially on each simple decomposition factor of $V$, $\bfU$ acts trivially on $V$. If $\bfU$ acts trivially on $V$, then the action of $\bfG$ is given by the action of $\bfL$, and so $V$ is semisimple.
\end{proof}

\begin{corollary}\label{cor:G socle Ta}
    The socle of $T_{\underline{a}}$ is $K_{\underline{a}}$.
\end{corollary}

\begin{proof}
    By Lemma~\ref{lem:homG Ta}, $K_{\underline{a}}$ is semisimple. Any submodule of $T_{\underline{a}}$ properly containing $K_{\underline{a}}$ has a nontrivial $\bfU$-action, and so by Corollary~\ref{cor:ss Gmod}, $K_{\underline{a}}$ is the maximal semisimple submodule of $T_{\underline{a}}$.
\end{proof}

\subsection{Injectives and projectives}

Let $\underline{\lambda}$ be an $n$-tuple of partitions and let $\underline{a}=(\abs{\lambda^1},\ldots,\abs{\lambda^n})$.
Define the $\bfG$-module $T_{\underline{\lambda}}$ by $T_{\underline{\lambda}}=\Hom_{\frakS_{\underline{a}}}(\bfM_{\underline{\lambda}},T_{\underline{a}}).$
This is therefore equal to
\[T_{\underline{\lambda}}=\bfS_{\lambda^1}(\bfV)\otimes\bfS_{\lambda^2}(\bfV/\bfV_1)\otimes\cdots\otimes\bfS_{\lambda^n}(\bfV/\bfV_{n-1}).\]
Similarly, define the $\bfG$-module $U_{\underline{\lambda}}$ by $U_{\underline{\lambda}}=\Hom_{\frakS_{\underline{a}}}(\bfM_{\underline{\lambda}},U_{\underline{a}})$, so it is equal to
\[U_{\underline{\lambda}}=\bfS_{\lambda^1}(\bfV_1)\otimes\bfS_{\lambda^2}(\bfV_2)\otimes\cdots\otimes\bfS_{\lambda^n}(\bfV_n).\]
It will follow from the equivalence of categories $\Rep(\bfG)\cong\Mod_{\scrU}$ that the $T_{\underline{\lambda}}$'s constitute a complete irredundant set of indecomposable injectives and that the $U_{\underline{\lambda}}$'s constitute a complete irredundant set of indecomposable projectives.

\section{Equivalence of the categories}\label{sec:equiv}

In this section, we prove the main theorem (Theorem~\ref{thm:main}) on the equivalence of tensor categories. We then give some consequences of the equivalence, including computing $\Ext^1$ groups among the simple objects and describing a universal property.

\subsection{Equivalence between $\Rep(\bfG)$ and $\Mod_{\scrU}$}\label{subsec:eq rep(G) Mod_U}

Recall the representations $T_{\underline{a}}\in\Rep(\bfG)$ defined in \S\ref{subsec:G prin obs}. They define a functor $\mathscr{T}:\scrU\to\Rep(\bfG)$ given by $\underline{a}\mapsto T_{\underline{a}}$.

\begin{proposition}
    We have a covariant equivalence of tensor categories $\Mod_{\scrU}\to\Rep(\bfG)$.
\end{proposition}

\begin{proof}
    We first show that $\Mod_{\scrD}$ is covariantly equivalent to $\Rep(\bfG)$. Recall from \S\ref{subsec:structured hom} that the structured tensor product $M\mapsto M\odot\mathscr{T}$ defines a tensor functor $\Mod_{\scrD}\to\Rep(\bfG)$. By \cite[Corollary 2.1.12]{SSstabpatterns}, to show that this functor is an equivalence, it suffices to prove the following:
    \begin{enumerate}
        \item For any $\underline{\lambda}$ with $\underline{a}=(\abs{\lambda^i})$, the polynomial $\bfG$-representation $\Hom_{\frakS_{\underline{a}}}(\bfM_{\underline{\lambda}},K_{\underline{a}})$ is simple.
        \item For every simple object $S_{\underline{\lambda}}$ of $\Rep(\bfG)$, there exists a unique tuple $\underline{a}$ such that $\Hom_{\bfG}(S_{\underline{\lambda}},T_{\underline{a}})$ is nonzero.
    \end{enumerate}
    By Lemma~\ref{lem:homG Ta},
    \[\Hom_{\frakS_{\underline{a}}}(\bfM_{\underline{\lambda}},K_{\underline{a}})=\bfS_{\lambda^1}(\bfV_{(1)})\otimes\cdots\otimes\bfS_{\lambda^n}(\bfV_{(n)})=S_{\underline{\lambda}}.\]
    By Lemma~\ref{cor:G socle Ta}, $S_{\underline{\lambda}}$ is a submodule of $T_{\underline{a}}$ if and only if $\underline{a}=(\abs{\lambda^1},\ldots,\abs{\lambda^n})$.

    The covariant equivalence $\Mod_{\scrU}\cong\Rep(\bfG)$ now follows from Proposition~\ref{prop:ModFBn selfdual}.
\end{proof}

We explicitly describe the equivalence using the constructions from \S\ref{subsec:structured hom}. The equivalence with $\Mod_{\scrD}$ is given by the functors
\begin{align*}
    &\Mod_{\scrD}\to\Rep(\bfG):M\mapsto \Hom(M^\vee,\mathscr{T})=M\odot\mathscr{T},\\
    &\Rep(\bfG)\to\Mod_{\scrD}:V\mapsto \Hom_{\bfG}(V,\mathscr{T})^\vee.
\end{align*}
Let $\mathcal{F}:\Mod_{\scrU}\to\Rep(\bfG)$ and $\mathcal{G}:\Rep(\bfG)\to\Mod_{\scrU}$ be the composition of these functors with $\tau_!$ and $\tau^*$ respectively. Then $\mathcal{F},\mathcal{G}$ give the equivalence $\Mod_{\scrU}\cong\Rep(\bfG)$.

\begin{lemma}\label{lem:structured hom prin inj}
    For any $i=1,\ldots,n$, let $e_i$ denote the tuple with $1$ in the $i^{\text{th}}$ position and $0$'s elsewhere. Then there is a natural isomorphism $\Hom(\bfI_{e_i},\mathscr{T})=\bfV_i$.
\end{lemma}

\begin{proof}
    Let $V$ be an object of $\Rep(\bfG)$. By the mapping property, giving a map $V\to\Hom(\bfI_{e_i},\mathscr{T})$ is the same as giving maps $f_j:V\to\Hom(\bfI_{e_i}(e_j),\bfV/\bfV_{j-1})$ for $j=1,\ldots,n$ such that the necessary relations hold. In particular, the maps $f_j$ are determined by $f_1$, and since $\bfI_{e_i}(e_j)=0$ for $j>i$, giving the map $f_1$ is the same as giving a map $V\to\ker(\bfV\twoheadrightarrow\bfV/\bfV_i)=\bfV_i$. Thus, $\bfV_i$ satisfies the same universal property as $\Hom(\bfI_{e_i},\mathscr{T})$, which gives the desired identification.
\end{proof}

\begin{lemma}
    For $\underline{a},\underline{\lambda}$, we have
    \[\mathcal{F}(\bfI_{\underline{a}})=T_{\tau(\underline{a})},\quad \mathcal{F}(\bfI_{\underline{\lambda}})=T_{\tau(\underline{\lambda})},\quad \mathcal{F}(\P_{\underline{a}})=U_{\tau(\underline{a})},\quad\mathcal{F}(\P_{\underline{\lambda}})=U_{\tau(\underline{\lambda})}.\]
\end{lemma}

\begin{proof}
We have that $\mathcal{F}(\bfI_{\underline{a}})=\Hom(\P_{\tau(\underline{a})},\mathscr{T})=\mathscr{T}(\tau(\underline{a}))=T_{\tau(\underline{a})}$ by Lemma~\ref{lem:structured hom prin proj}. We also have that $\mathcal{F}(\P_{e_i})=\Hom(\bfI_{e_{n-i+1}},\mathscr{T})=\bfV_{n-i+1}$ by Lemma~\ref{lem:structured hom prin inj}. Then since $\P_{\underline{a}}=\bigotimes_{j=1}^n\P_{e_j}^{\otimes a_j}$
and $\mathcal{F}$ is a tensor functor, we see that $\mathcal{F}(\P_{\underline{a}})=U_{\tau(\underline{a})}$.

The other equalities follow from taking $\bfM_{\underline{\lambda}}$-isotypic pieces of $\bfI_{\underline{a}}$ and $\mathcal{F}(\bfM_{\underline{\lambda}})=S_{\tau(\underline{\lambda})}$-isotypic pieces of $T_{\tau(\underline{\lambda})}$, and similarly for the projective objects.
\end{proof}

\subsection{Category of polynomial functors on flags}\label{subsec:cat A}
Let $\mathscr{A}$ denote the category of polynomial functors $\Flag_n\to\Vec$. In this subsection, we show that objects of $\mathscr{A}$ correspond to sequences of representations of parabolic subgroups of general linear groups that stabilize. This will be used to show the equivalence $\mathscr{A}\cong\Rep(\bfG)$.

Suppose that $F\in\mathscr{A}$. For any tuple $\underline{a}$ of nonnegative integers, let $\C^{\underline{a}}$ denote the following flag:
\[0\subset \C^{a_1}\subset\C^{a_1+a_2}\subset\cdots\subset\C^{\abs{\underline{a}}}.\]
Note that every object in $\Flag_n$ is isomorphic to a unique such $\C^{\underline{a}}$. We have that $F(\C^{\underline{a}})$ is a representation of the parabolic subgroup $G_{\underline{a}}\subset\GL(\C^{\abs{\underline{a}}})$ of $n\times n$ block upper triangular matrices, where the $(i,j)$-block has dimension $a_i\times a_j$. Note that this group is isomorphic to the group $G_{\underline{a}}$ defined in \S\ref{subsec:G invar}. Furthermore, if $F\to F'$ is a morphism in $\mathscr{A}$, then the linear map $F(\C^{\underline{a}})\to F'(\C^{\underline{a}})$ is $G_{\underline{a}}$-equivariant.

If $\underline{a},\underline{b}\in\N^n$ with $a_i\leq b_i$ for all $i$, there is a natural morphism $\iota:\C^{\underline{a}}\to\C^{\underline{b}}$ in $\Flag_n$ given by inclusion on each graded piece of the flags. Notice that this morphism splits, and so if $F\in\mathscr{A}$, then $F(\iota):F(\C^{\underline{a}})\to F(\C^{\underline{b}})$ is an injective $G_{\underline{a}}$-equivariant linear map, where $G_{\underline{a}}\subset G_{\underline{b}}$ is naturally a subgroup via $\iota$. Such maps form a directed system, and so one can consider the direct limit $\varinjlim F(\C^{\underline{a}})$. For any $j\gg 0$, there exists one of these natural injections $\C^{\underline{a}}\to\C^{(j^n)}$, where $(j^n)=(j,\ldots,j)$; therefore, one can equivalently consider the direct limit $\varinjlim F(\C^{(j^n)})$.

\begin{lemma}\label{lem:dirlim poly rep}
    If $F\in\mathscr{A}$, then $\varinjlim F(\C^{(j^n)})$ is a polynomial representation of $\bfG$.
\end{lemma}

\begin{proof}
    Recall that $F_a\in\mathscr{A}$ is the functor $\{V_i\}\mapsto V^{\otimes a}$. We have that $\varinjlim F_a(\C^{(j^n)})\cong\bfV^{\otimes a}$. The result then follows from direct limits being exact.
\end{proof}

\subsection{Equivalence between $\Rep(\bfG)$ and $\mathscr{A}$}

We now show that $\Rep(\bfG)\cong\mathscr{A}$. By Lemma~\ref{lem:dirlim poly rep}, there is a functor $\Phi:\mathscr{A}\to\Rep(\bfG)$ given by $F\mapsto\varinjlim F(\C^{(j^n)})$. Note that $\Phi$ is a tensor functor since it is defined using a direct limit.

We now define a functor $\Psi:\Rep(\bfG)\to\mathscr{A}$. Recall the subgroups $H_{\underline{a}}$ and $G_{\underline{a}}$ of $\bfG$ defined in \S\ref{subsec:G invar}.
For a polynomial $\bfG$-representation $V$, define $\Psi(V)(\C^{\underline{a}})=V^{H_{\underline{a}}}$ to be the $H_{\underline{a}}$-invariants of $V$. This is a representation of $G_{\underline{a}}$, since $G_{\underline{a}}$ commutes with $H_{\underline{a}}$. As with $\Phi$, this functor takes polynomial $\bfG$-representations to polynomial functors. Suppose $f:\C^{\underline{a}}\to\C^{\underline{b}}$ is a morphism in $\Flag_n$. There is an element $g_f$ in $\End(\bfV)$ extending $f$ such that $g_f$ restricts to $f$ on $\C^{\underline{a}}$ and such that for any $h\in H_{\underline{b}}$, there exists an $h'\in H_{\underline{a}}$ such that the action of $h\circ g_f$ on the standard representation $\bfV$ is equal to the action of $g_f\circ h'$. This shows that if $V$ is a polynomial representation of $\bfG$, the image of $V^{H_{\underline{a}}}$ from the induced action of $g_f$ on $V$ is contained in $V^{H_{\underline{b}}}$. Define the morphism $\Psi(V)(f)$ in $\Vec$ by the action of $g_f$ on $V$. By Lemma~\ref{lem:invar tensor}, $\Psi$ is a tensor functor.

\begin{lemma}
    $(\Phi,\Psi)$ is an adjoint pair and the counit $\Phi\Psi\to\id$ is an equality.
\end{lemma}

\begin{proof}
    Let $F$ be an object of $\mathscr{A}$, and $V$ an object of $\Rep(\bfG)$. Suppose $f:\Phi(F)\to V$ is a $\bfG$-equivariant map. For any $\underline{a}$, there is a canonical $G_{\underline{a}}$-equivariant map $F(\C^{\underline{a}})\to \Phi(F)$, and composing with $f$ and then taking $H_{\underline{a}}$-invariants induces a $G_{\underline{a}}$-equivariant map $f_{\underline{a}}:F(\C^{\underline{a}})\to V^{H_{\underline{a}}}$. These maps are compatible with all maps induced by morphisms in $\Flag_n$ since $f$ is equivariant under the action of $\End(\bfV)$. We therefore obtain a morphism $F\to\Psi(V)$ in $\mathscr{A}$.

    Conversely, suppose $g:F\to\Psi(V)$ is a map in $\mathscr{A}$, and apply $\Phi$ to obtain $\Phi(g):\Phi(F)\to\Phi(\Psi(V))$. Note that $\Phi(\Psi(V))=V$, since every element of $V$ is invariant under $H_{(j^n)}$ for $j\gg 0$. This gives a map $\Phi(F)\to V$, and one can check that this construction is inverse to the one above. The claim that the counit is an equality also follows.
\end{proof}

\begin{proposition}
    The functors $\Phi,\Psi$ give an equivalence of tensor categories $\mathscr{A}\cong\Rep(\bfG)$.
\end{proposition}

\begin{proof}
    It suffices to show that both $\Phi,\Psi$ are fully faithful. The isomorphism $\Phi\Psi\to\id$ shows that $\Psi$ is fully faithful.

    We claim $\Phi$ is faithful. Fix $\underline{a}$, and let $F\in\mathscr{A}$. Then for any $j\gg 0$, there is a natural morphism $\iota_j:\C^{\underline{a}}\to\C^{(j^n)}$ in $\Flag_n$ that splits, and so $F(\iota_j)$ is injective. Furthermore, the $\iota_j$ maps are equivariant with the inclusions $\C^{(j^n)}\to\C^{((j+1)^n)}$. We therefore have an injection $F(\C^{\underline{a}})\to\Phi(F)=\varinjlim F(\C^{(j^n)})$.
    Now, suppose $f:F\to G$ is a morphism in $\mathscr{A}$. We then obtain the following commutative square with vertical maps that are injections:
    \[
 \xymatrix@C+0.25em@R+0.25em{ 
 	 F(\C^{\underline{a}}) \ar^{f(\C^{\underline{a}})}[r] \ar[d] & G(\C^{\underline{a}}) \ar^{}[d]  \\
 	\Phi(F) \ar^{\Phi(f)}[r] & \Phi(G)
 }
 \]
Thus, if $\Phi(f)=0$, then $f=0$ and so $\Phi$ is faithful.

We now show $\Phi$ is full. Suppose $f:M\to N$ is a map in $\Rep(\bfG)$, and let $g=\Psi(f)$. Then $\Phi(g)=\Phi\Psi(f)=f$, since $\Phi\Psi=\id$. 
\end{proof}

\subsection{Consequences of the equivalence}

We now list some immediate consequences of the equivalence of categories.

\begin{proposition}
Consider the category $\mathscr{A}$ of polynomial functors $\Flag_n\to\Vec$. Let $\{V_i\}$ denote the object $V$ in $\Flag_n$ with flag $0=V_0\subset V_1\subset\cdots \subset V_n=V$.
    \begin{enumerate}
        \item The functor $F_{\underline{\lambda}}:\{V_i\}\mapsto\bigotimes_{i=1}^n \bfS_{\lambda^i}(V_i/V_{i-1})$ is simple, and such functors constitute a complete irredundant set of simple objects in $\mathscr{A}$.
        \item The injective envelope and projective cover of $F_{\underline{\lambda}}$ are
        \[\{V_i\}\mapsto \bigotimes_{i=1}^n\bfS_{\lambda^i}(V_n/V_{i-1}),\qquad\{V_i\}\mapsto\bigotimes_{i=1}^n\bfS_{\lambda^i}(V_i).\]
        \item Polynomial functors $\Flag_n\to \Vec$ are of finite length, and have finite injective and projective dimensions.
        \item $\mathscr{A}$ is self-dual.
        \end{enumerate}
\end{proposition}

\begin{proposition}
Consider the category $\Rep(\bfG)$ of polynomial representations of $\bfG$.
\begin{enumerate}
    \item The $\bfG$-modules $T_{\underline{a}}$ and $T_{\underline{\lambda}}$ are injective.
    \item $T_{\underline{\lambda}}$ is the injective envelope of $S_{\underline{\lambda}}$, and the $T_{\underline{\lambda}}$'s constitute a complete irredundant set of indecomposable injective objects in $\Rep(\bfG)$.
    \item The $\bfG$-modules $U_{\underline{a}}$ and $U_{\underline{\lambda}}$ are projective.
    \item $U_{\underline{\lambda}}$ is the projective cover of $S_{\underline{\lambda}}$, and the $U_{\underline{\lambda}}$'s constitute a complete irredundant set of indecomposable projective objects in $\Rep(\bfG)$.
    \item All objects are of finite length, and have finite injective and projective dimensions.
    \item Tensor powers of the standard representation $\bfV^{\otimes a}$ are both projective and injective.
    \item The category $\Rep(\bfG)$ is self-dual.
\end{enumerate}
\end{proposition}

\subsection{Extensions of simples}

In this subsection, we compute the first Ext groups between simples objects.

\begin{lemma}\label{lem:ext nonzero necc}
    Let $\underline{\lambda},\underline{\mu}$ be tuples of partitions with $\underline{a}=(\abs{\mu^i})$ and $\underline{b}=(\abs{\lambda^i})$.
    If $\Ext^1_\bfG(S_{\underline{\lambda}},S_{\underline{\mu}})\neq 0$, then it must be that $\abs{\underline{a}}=\abs{\underline{b}}$ and $\underline{a}>\underline{b}$ is a cover relation.
\end{lemma}

\begin{proof}
    Recall from (\ref{eqn:ses inj}) that we have the exact sequence
\[0\to K_{\underline{a}}\to T_{\underline{a}}\to\bigoplus_{f:\underline{a}\to\underline{c}}T_{\underline{c}},\]
where $\abs{\underline{a}}=\abs{\underline{c}}$ with $\underline{a}>\underline{c}$ a cover relation. Then, take the $\bfM_{\underline{\mu}}$-isotypic piece of this sequence to obtain the exact sequence $0\to S_{\underline{\mu}}\to T_{\underline{\mu}}\to T$, where $T$ is a summand of $\bigoplus T_{\underline{c}}$. In particular, $T$ is injective and if $T_{\underline{\lambda}}$ is a summand of $T$, then it must be that $\abs{\underline{a}}=\abs{\underline{b}}$ and $\underline{a}>\underline{b}$ is a cover relation. 
\end{proof}

We now give the formula for the first Ext groups by examining projective $\scrD$-modules. For partitions $\lambda,\mu$, we say that $\lambda/\mu\in\text{HS}_1$ if $\mu\subset\lambda$ as Young diagrams and their difference is a single box.

\begin{proposition}
 Let $\underline{\lambda},\underline{\mu}$ be tuples of partitions with $\underline{a}=(\abs{\mu^i})$ and $\underline{b}=(\abs{\lambda^i})$.
    We have $\Ext^1_\bfG(S_{\underline{\lambda}},S_{\underline{\mu}})\neq 0$ if and only if 
    \begin{enumerate}
        \item $\abs{\underline{a}}=\abs{\underline{b}}$ and $\underline{a}>\underline{b}$ is a cover relation where $a_i-1=b_i$,
        \item $\mu^j=\lambda^j$ for $j\neq i,i+1$, and
        \item $\mu^i/\lambda^i$ and $\lambda^{i+1}/\mu^{i+1}$ are both in $\text{HS}_1$.
    \end{enumerate}
    In this case, its dimension is $1$.
\end{proposition}

\begin{proof}
   First suppose $\underline{\lambda},\underline{\mu}$ satisfy the three conditions. Consider the simple $\scrD$-module $\bfM_{\underline{\lambda}}$, which has projective cover $\bfQ_{\underline{\lambda}}$. Then, if $Q_\bullet\to \bfM_{\underline{\lambda}}\to 0$ is the projective resolution of $\bfM_{\underline{\lambda}}$, we have that $Q_0=\bfQ_{\underline{\lambda}}$. We now determine the multiplicity of $\bfQ_{\underline{\mu}}$ in $Q_1$.
    
    We have that the kernel of $\bfQ_{\underline{\lambda}}\twoheadrightarrow\bfM_{\underline{\lambda}}$ evaluated on the weighted set $\underline{a}$ is given by first restricting $\bfM_{\underline{\lambda}}$ as a representation of $\frakS_{\underline{b}}$ to a representation of the subgroup
    \[\frakS_{b_1}\times\cdots\times\frakS_{b_i}\times\frakS_1\times\frakS_{b_{i+1}-1}\times\frakS_{b_{i+2}}\times\cdots\times\frakS_{b_n},\]
    via $\frakS_{b_{i+1}-1}\times\frakS_1\subset\frakS_{b_{i+1}}$,
    and then inducting to $\frakS_{\underline{a}}$ via $\frakS_{b_{i}}\times\frakS_1\subset\frakS_{a_{i}}$. By Pieri's rule, $\bfM_{\underline{\mu}}$ appears with multiplicity one. Since $\underline{a}>\underline{b}$ is a cover relation, the projective cover of $\bfM_{\underline{\mu}}$ must appear in $Q_1$, and it appears with multiplicity one. Therefore, $\Ext^1_\scrD(\bfM_{\underline{\lambda}},\bfM_{\underline{\mu}})=\Ext^1_\bfG(S_{\underline{\lambda}},S_{\underline{\mu}})$
    has dimension $1$.

    By Lemma~\ref{lem:ext nonzero necc}, if condition (1) is not satisfied, then $\Ext^1_\bfG(S_{\underline{\lambda}},S_{\underline{\mu}})=0$. Now, if condition (1) holds but conditions (2) and (3) are not satisfied, then $\bfM_{\underline{\mu}}$ does not appear in the decomposition of $\bfQ_{\underline{\lambda}}(\underline{a})$, and so there is no map from $\bfQ_{\underline{\mu}}\to\bfQ_{\underline{\lambda}}$. Therefore, $\Ext^1_\bfG(S_{\underline{\lambda}},S_{\underline{\mu}})=0$ in this case as well.
\end{proof}

An interesting problem is to construct the minimal projective or injective resolution for simple objects in our category, and to then compute all higher Ext groups among simple objects. We expect any formulas for such groups will involve Littlewood--Richardson coefficients.

\subsection{Universal property}
Using arguments similar to those in \cite[\S 3.4]{SSstabpatterns}, we now describe left-exact tensor functors from the category $\Mod_\scrD$ to an arbitrary abelian tensor category $\mathscr{C}$. By pre-composing with the pushforward $\tau_!:\Mod_\scrU\to\Mod_\scrD$, this also describes left-exact tensor functors from $\Mod_\scrU$. The model for these functors is the one $\Mod_\scrD\to\Rep(\bfG)$ described in \S\ref{subsec:eq rep(G) Mod_U} given by the structured tensor product with the functor $\mathscr{T}$.

Let $T(\mathscr{C})$ be the category whose objects consist of an object $A$ of $\mathscr{C}$ along with a sequence of morphisms
\[A=A_1\to A_2\to\cdots\to A_n.\]
We denote this object by $\{A_i\}$. Given $\{A_i\}\in T(\mathscr{C})$, define $\mathcal{K}(\{A_i\})$ to be the functor $\scrU\to\mathscr{C}$ given by
\[\underline{a}\mapsto\bigotimes_{i=1}^nA_i^{\otimes a_i}.\]
A morphism $\varphi:\underline{a}\to\underline{b}$ in $\scrU$ defines a morphism in $\mathscr{C}$ given by mapping the tensor factors according to $\varphi$ and using the morphisms specified by $\{A_i\}$.
One should view this as a generalization of the functor $\mathscr{T}:\scrU\to\Rep(\bfG)$ giving injective objects in $\Rep(\bfG)$, where the sequence of morphisms are the surjections
\[\bfV\to\bfV/\bfV_1\to\cdots\to\bfV/\bfV_{n-1}.\]

\begin{proposition}
    To give a left-exact tensor functor $\Mod_\scrD\to\mathscr{C}$ is equivalent to giving an object of $T(\mathscr{C})$.
\end{proposition}

\begin{proof}
    Let $\text{LEx}(\Mod_\scrD,\mathscr{C})$ denote the category of left-exact tensor functors $\Mod_\scrD\to\mathscr{C}$. We have a functor $\Phi:\text{LEx}(\Mod_\scrD,\mathscr{C})\to T(\mathscr{C})$ defined as follows. Let $F\in\text{LEx}(\Mod_\scrD,\mathscr{C})$. Recall that there are natural maps $\bfJ_{e_1}\to\bfJ_{e_2}\to\cdots\to\bfJ_{e_n}$, where $e_i$ denotes the $n$-tuple with $1$ in the $i^{\text{th}}$ position and $0$'s elsewhere.
    Then $\{F(\bfJ_{e_{i}})\}$ gives an object in $T(\mathscr{C})$. We also have a functor $\Psi:T(\mathscr{C})\to\text{LEx}(\Mod_\scrD,\mathscr{C})$ given by
    \[\{A_i\}\mapsto (N\mapsto N\odot \mathcal{K}(\{A_i\}).\]
    By the construction of the structured tensor product, it gives a left-exact tensor functor. We show that $\Phi$ and $\Psi$ are mutually quasi-inverse equivalences.

    Let $\{A_i\}$ be an object of $T(\mathscr{C})$. Then applying the functor $\Psi(\{A_i\})$ to $\bfJ_{e_i}$, we obtain 
    \[\bfJ_{e_i}\mapsto\bfJ_{e_i}\odot\mathcal{K}(\{A_i\})=\Hom(\P_{e_i},\mathcal{K}(\{A_i\}))=A_i.\]
    Thus, $\Phi\Psi(\{A_i\})=\{A_i\}$, and so the natural morphism $\id\to\Phi\Psi$ is an equality.

    Now suppose $F\in\text{LEx}(\Mod_\scrD,\mathscr{C})$. Applying $\Psi\Phi$, we obtain the functor $\Mod_{\scrD}\to\mathscr{C}$ given by
    \[N\mapsto N\odot\mathcal{K}(\{F(\bfJ_{e_{i}})\})=N\odot F(\mathcal{K}(\{\bfJ_{e_{i}}\}))=F(N\odot\mathcal{K}(\{\bfJ_{e_{i}}\})),\]
    where the first equality follows from $F$ being a tensor functor, and the second equality comes from $F$ being left-exact and $\odot$ being an inverse limit. It therefore suffices to show that $N\odot\mathcal{K}(\{\bfJ_{e_i}\})=N$.

    Let $M$ be a $\scrD$-module. Then
    \[\Hom_\scrD(M,N\odot\mathcal{K}(\{\bfJ_{e_{i}}\})=\Hom_\scrD(M,\Hom(N^\vee,\mathcal{K}(\{\bfJ_{e_{i}}\})))=\Hom_\scrU(N^\vee,\Hom_\scrD(M,\mathcal{K}(\{\bfJ_{e_{i}}\})));\]
    the second equality comes from the structured Hom being an adjoint functor, as described in \S\ref{subsec:structured hom}.
    We have that
    \[\Hom_\scrD(M,\mathcal{K}(\{\bfJ_{e_i}\})(\underline{a}))=\Hom_\scrD(M,\bfJ_{\underline{a}})=M^\vee(\underline{a}).\]
    In particular, $\Hom_\scrD(M,\mathcal{K}(\{\bfJ_{e_i}\}))=M^\vee$ as $\scrU$-modules, and so
    \[\Hom_\scrD(M,N\odot\mathcal{K}(\{\bfJ_{e_i}\}))=\Hom_\scrU(N^\vee,M^\vee)=\Hom_\scrD(M,N).\]
    Thus, we see that $N\odot\mathcal{K}(\{\bfJ_{e_i}\})=N$, and this completes the proof.
\end{proof}

\bibliographystyle{alpha}
\bibliography{references}

\end{document}